\documentclass[12pt]{amsart}
\usepackage[top=30truemm,bottom=30truemm,left=25truemm,right=25truemm]{geometry}
\usepackage{txfonts}
\usepackage{mathrsfs}

\usepackage{color}
\usepackage{bm}
\usepackage{amsfonts,amssymb}
\usepackage{dsfont}
\usepackage{amscd}
\usepackage{extarrows}
\usepackage{amsmath}
\usepackage{mathrsfs}
\usepackage{enumerate}
\usepackage{amscd}
\usepackage[all]{xy}
\usepackage[pagebackref,colorlinks]{hyperref}

\usepackage{extarrows}
\theoremstyle{plain} 
\newtheorem{theorem}{\indent\bf Theorem}[section]
\newtheorem{lemma}[theorem]{\indent\bf Lemma}

\newtheorem{proposition}[theorem]{\indent\bf Proposition}

\theoremstyle{definition} 
\newtheorem{definition}[theorem]{\indent\bf Definition}

\newcommand{\be}{\begin{equation}}
\newcommand{\ee}{\end{equation}}
\newcommand{\bea}{\begin{eqnarray}}
\newcommand{\eea}{\end{eqnarray}}
\newcommand{\ben}{\begin{eqnarray*}}
	\newcommand{\een}{\end{eqnarray*}}
\newcommand{\bt}{\begin{split}}
	\newcommand{\et}{\end{split}}
\newcommand{\bet}{\begin{equation}}

\begin{document}
\title{On a Bogomolov type vanishing theorem}

\author[Z. Li]{Zhi Li}
\address{Zhi Li: School of Science, Beijing University of Posts and Telecommunications, Beijing 100876, China.}
\email{lizhi@amss.ac.cn, lizhi10@foxmail.com}

\author[X. Meng]{Xiankui Meng\textsuperscript{*}}
\address{Xiankui Meng: School of Science, Beijing University of Posts and Telecommunications,
	Beijing 100876, China.}
\email{mengxiankui@amss.ac.cn}

\author[J. Ning]{Jiafu Ning\textsuperscript{*}}
\address{Jiafu Ning: \ Department of Mathematics, Central South University, Changsha, Hunan 410083, China.}
\email{jfning@csu.edu.cn}
\author[Z. Wang]{Zhiwei Wang}
\address{Zhiwei Wang: Laboratory of Mathematics and Complex Systems (Ministry of Education)\\ School of Mathematical Sciences\\ Beijing Normal University\\ Beijing 100875, China.}
\email{zhiwei@bnu.edu.cn}
\author[X. Zhou]{Xiangyu Zhou}
\address{Xiangyu Zhou: Institute of Mathematics\\Academy of Mathematics and Systems Sciences\\and Hua Loo-Keng Key
	Laboratory of Mathematics\\Chinese Academy of
	Sciences\\Beijing\\100190, China.}
\address{School of
	Mathematical Sciences, University of Chinese Academy of Sciences,
	Beijing 100049, China}
\email{xyzhou@math.ac.cn}
\begin{abstract}
	Let $X$ be a compact K\"ahler manifold, and let $L \rightarrow X$ be a holomorphic line bundle equipped with a singular metric $h$ such that the curvature $\mathrm{i}\Theta_{L,h}\geqslant 0$ in the sense of currents.  The main result of the present paper is that $H^n(X,\mathcal{O}(\Omega^p_X\otimes L)\otimes \mathcal{I}(h))=0$ for $p\geqslant n-\operatorname{nd}(L,h)+1$.  This is a generalization of Bogomolov's vanishing theorem.
\end{abstract}

\thanks{ This research is supported by National Key R\&D Program of China (No. 2021YFA1002600).  The authors are partially supported respectively by NSFC grants (12271057, 12201060, 12071485, 12071035, 11688101).
Z.W. Wang  is partially supported by Beijing Natural Science Foundation (1202012, Z190003). The second author and the third author are both corresponding authors.}
%
{\maketitle}

\section{Introduction}
Let $X$ be a complex projective manifold of dimension $n$, and let $L\rightarrow X$ be a holomorphic line bundle over $X$. The famous Bogomolov  vanishing theorem \cite{Bog78} says that
$$H^0\left(X,\mathcal{O}\left(\Omega^p_X\otimes L^{-1}\right)\right)=0,$$
for $p<\kappa(L)$, where $\kappa(L)$ is the Kodaira-Iitaka dimension of $L$. By the Serre duality theorem, it is equivalent to that
$$H^n\left(X,\mathcal{O}\left(\Omega^p_X\otimes L\right)\right)=0,$$
for $p\geqslant n-\kappa(L)+1$.

There are many generalizations of the Bogomolov vanishing theorem.
In \cite{Mou98},  Mourougane proved the following two vanishing theorems:
\begin{itemize}
\item if $L$ is a nef line bundle over a compact K\"ahler manifold $X$, then
$$H^0\left(X,\mathcal{O}\left(\Omega^p_X\otimes L^{-1}\right)\right)=0,$$
for $p<\operatorname{nd}(L)$, where $\operatorname{nd}(L)=\max\{k: ({c_{1}(L)})^k\neq 0\}$.
\item If $L$ is a pseudoeffective over a compact K\"ahler manifold $X$, then
$$H^0\left(X,\mathcal{O}\left(\Omega^p_X\otimes L^{-1}\right)\right)=0,$$
for $p<e(L)$, where $e(L)$ is defined to be the largest natural number $k$ such that there exists a positive $(1,1)$-current $T\in c_1(L)$ whose absolutely continuous part in the Lebesgue decomposition has rank $k$ on a strictly positive Lebesgue measure set on $X$.
\end{itemize}
In \cite{Bou02}, Bouksom proved that if $L$ is a pseudoeffective line bundle over a compact K\"ahler manifold $(X,\omega)$, then
$$H^0\left(X,\mathcal{O}\left(\Omega^p_X\otimes L^{-1}\right)\right)=0,$$
if $p<\operatorname{nd}(L)$, where $\operatorname{nd}(L)$ here is defined (cf.  \cite{Bou02, Dem12}) to be the largest integer $k$, such that $\exists c>0, \forall \varepsilon>0, \exists $ Hermitian metrics $h_\varepsilon$ of $L$ with analytic singularities,
such that $\mathrm{i}\Theta_{L,h_\varepsilon}\geqslant -\varepsilon\omega$, and $\int_{X\setminus Z_\varepsilon}(\mathrm{i}\Theta_{L,h_\varepsilon}+\varepsilon\omega)^k\wedge \omega^{n-k}\geqslant c$, where  $Z_\varepsilon$ is the singular set of $h_\varepsilon$. It is proved in \cite{Bou02} that if $L$ is nef, the two numerical dimensions mentioned above  coincide. For a new proof, see also \cite{Wu20}.

Multiplier ideal sheaves play an important role in the study of the several complex variables and complex geometry. Inspired by Nadel's vanishing theorem, it is natural to ask if one can get Bogomolov's vanishing theorem with multiplier ideal sheaves.

Recall that a singular (Hermitian) metric on a line bundle is simply a Hermitian metric which is given in any trivialization by a weight function $e^{-\varphi}$ such that $\varphi$ is locally integrable.  Recently, Watanabe \cite{Wat22} proved that  if $L$ is a holomorphic line bundle over a projective manifold $X$ of dimension $n$ equipped with a K\"ahler metric $\omega$, and if $h$ is a singular Hermitian metric of $L$ such that $\mathrm{i}\Theta_{L,h}\geqslant \varepsilon\omega$ in the sense of currents for some $\varepsilon>0$, then
$$H^n\left(X,\mathcal O\left(\Omega^p_X\otimes L\right)\otimes \mathcal I\left(h\right)\right)=0,$$
for $p>0$, where $\mathcal I(h)$ is the multiplier ideal sheaf associated to the singular metric $h$.

Note that in Watanabe's theorem, the line bundle is assumed to be big, and the manifold is assumed to be projective.

In this paper, we  prove the following Bogomolov type vanishing theorem for compact K\"ahler manifolds and pseudoeffective line bundles.
\begin{theorem}
\label{thm:main}Let $X$ be a compact K\"ahler manifold of dimension $n$. Let $L \rightarrow X$ be a holomorphic line
bundle equipped with a singular metric $h$ with
\[ \mathrm{i}\Theta_{L, h}\geqslant 0, \]
in the sense of currents, and let $\operatorname{nd} (L, h)$ denote the numerical
dimension of $(L, h)$. Then
\[ H^n\left(X,\mathcal O\left(\Omega^p_X\otimes L\right)\otimes \mathcal I(h)\right)=0, \]
for $p \geqslant n - \operatorname{nd}(L, h) + 1$.
\end{theorem}
For the definition of  $\operatorname{nd}(L,h)$, see Definition \ref{def}, which is originally defined in \cite{cao2014}. In the same paper, an example is provided to show that in general, one does not have  $\operatorname{nd}(L,h)=\operatorname{nd}(L)$, even if $h$ has minimal singularities, cf. \cite[Example 1.7]{cao2014}.

As an application, one can find a singular metric $h$ on $L$ through the Kodaira-Iitaka map, such that $\kappa(L)\leqslant \operatorname{nd} (L,h)$,  thus our  Theorem \ref{thm:main} covers the Bogomolov vanishing theorem (see Theorem \ref{thm:bogomolov}). Meanwhile, if $\mathrm{i}\Theta_{L,h}\geqslant \varepsilon\omega$ for some $\varepsilon>0$, then $\operatorname{nd}(L,h)=n$, thus our Theorem \ref{thm:main} also implies Watanabe's result.

The structure of the present paper is as follows. In \S \ref{sect: nd}, we introduce the numerical dimension $\operatorname{nd}(L,h)$. In \S \ref{sect: l2}, we prove an $L^2$-estimate for the $\bar\partial$-equation related to the numerical dimension $\operatorname{nd}(L,h)$. In \S \ref{sect: cech}, we prepare a criterion of the vanishing of cohomology class in the \v Cech cohomology group. In \S \ref{sect:proof}, we complete the proof of the Theorem \ref{thm:main}.

\section{Numerical dimension}\label{sect: nd}
Let $\theta+dd^c\varphi$ be  a positive closed current on a compact K\"ahler manifold $(X,\omega)$, where $\theta$ is a smooth form and $\varphi$ is a quasi-plurisubharmonic function on $X$. From \cite[Theorem 2.2.1]{DPS01}, there is a quasi-equisingular approximation $\{\varphi_k\}$ of $\varphi$ for the current $\theta+dd^c\varphi$, such that
\begin{itemize}
\item the sequence $\{\varphi_k\}$ converges to $\varphi$ in $L^1$ topology and $\theta+dd^c\varphi_k\geqslant -\tau_k\omega$ for some constants $\tau_k\rightarrow 0$ as $k\rightarrow \infty$;
\item all the $\varphi_k$ has analytic singularities and $\varphi_k$ is less singular than $\varphi_{k+1}$, i.e. $\varphi_{k+1}\leqslant \varphi_{k}+O(1)$;
\item for any $\delta>0$ and $m\in \mathbb N$, there exists $k_0(\delta,m)\in \mathbb N$ such that
$\mathcal I(m(1+\delta)\varphi_k)\subset \mathcal I(m\varphi)$ for every $k\geqslant k_0(\delta,m)$.
\end{itemize}
Let $T_1=\theta_1+dd^c\varphi_1,\cdots, T_k=\theta_k+dd^c\varphi_k$ be closed positive $(1,1)$-currents on a compact K\"ahler manifold $X$. In \cite{cao2014}, it is proved that one can define  a cohomological  product
$$\langle T_1\wedge\cdots\wedge T_k\rangle\in H^{k,k}_{\geqslant 0}(X),$$
such that for all $u\in H^{n-k,n-k}(X)$,
$$\langle T_1\wedge\cdots\wedge T_k\rangle\wedge u \coloneqq \lim_{j\rightarrow \infty}\int_X(\theta_1+dd^c\varphi_{1,j})_{ac}\wedge\cdots\wedge(\theta_k+dd^c\varphi_{k,j})_{ac}\wedge u,$$
where $\{\varphi_{i,j} \}_{j=1}^{\infty}$ is a quasi-equisingular approximation of $\varphi_i$ and $(\theta_i+dd^c\varphi_{i,j})_{ac}$ is the absolute continuous part of the Lebesgue decomposition of the  current $\theta_i+dd^c\varphi_{i,j}$.
The key point of the product is that the above limit exists and does not depend on the choice of the quasi-equisingular approximation.
\begin{definition}[\cite{cao2014}]\label{def}
Let $L$ be a holomorphic line bundle on a compact K\"ahler manifold $X$ and $h$ a singular metric on $L$. The numerical dimension $\operatorname{nd}(L,h)$ is defined to be the largest $v\in \mathbb N$ such that  $\langle (\mathrm{i}\Theta_{L,h})^v\rangle\neq 0$.
\end{definition}

Another important concept is the Kodaira-Iitaka dimension. Let $L$ be a holomorphic line bundle over a compact complex manifold. For each positive integer $m\geqslant 1$, let $s_0, \cdots, s_N$ be a basis of $H^0(X,mL)$. Consider the Kodaira-Iitaka map
\[
\Phi_{m}: X \setminus B_{m} \longrightarrow \mathbb{P}^N = P \left(H^0 (X, m L)^{\ast}\right), \quad x\mapsto[s_0(x):\cdots:s_N(x)],
\]
defined by $s_0, \cdots, s_N$, where $B_{m} = \bigcap_{i=0}^{N} s_{i}^{-1}(0)$ is the base locus of the linear system $| mL |$. There is a canonical commutative diagram
\begin{equation*}
\begin{CD}
\left. L \right|_{X \setminus B_{m}} @>>> \mathcal{O}(1) \\
@VVV @VVV\\
X \setminus B_{m} @>\Phi_{m}>> P \left(H^0 (X, m L)^{\ast}\right).
\end{CD}
\end{equation*}
Moreover, the holomorphic map $\Phi_{m}$ can be extended to a meromorphic map $\Phi_{m}: X \dashrightarrow \mathbb{P}^{N}$. Denote by $Y_{m} \subset \mathbb{P}^{N}$ the image of the meromorphic map $\Phi_m$. The Kodaira-Iitaka dimension $\kappa(L)$ of $L$ is defined to be
\[ \sup\left\{\dim Y_m:m\geqslant 1\right\}. \]
In the case $H^0(X,mL)=0$, we set $\kappa(L)=-\infty$. It is clear that $\kappa(L)\leqslant \dim X$.


\begin{lemma}\label{lem:ki}
  Let $X$ be a compact K\"ahler manifold and $L \rightarrow X$ a holomorphic line bundle over $X$. Then there exists a singular metric $h$ on $L$ such that the Kodaira dimension $\kappa (L) \leqslant \operatorname{nd} (L, h)$.
\end{lemma}

\begin{proof}
With the notations as above, we may assume that $\kappa ( L ) = \dim Y_{m} \geqslant 0$ for some $m \geqslant 1$. Denote by $j : Y_{m} \rightarrow \mathbb{P}^{N}$ the natural inclusion map. Then $\Phi_{m}$ can be written as the composition
\[ X \setminus B_{m} \xrightarrow{\Psi_{m}} Y_{m} \xrightarrow{j} \mathbb{P}^{N}, \]
where $\Psi_{m}$ is induced by $\Phi_{m}$. In this case, we have
\[
\left. L \right|_{X \setminus B_{m}} = \Phi_{m}^{\ast} \mathcal{O} ( 1 ) = \Psi_{m}^{\ast} \left( j^{\ast} \mathcal{O} ( 1 ) \right).
\]
  Denote by $h_{\mathrm{can}}$ the metric on $\mathcal{O} (1)$ defined by
  \[ e^{- \log (| z_0 |^2 + \cdots + | z_N |^2)}, \]
  where $[z_0: \cdots: z_N]$ is the homogeneous coordinate of $\mathbb{P}^N$.
  Its curvature $\mathrm{i}\Theta_{\mathcal{O} (1), h_{\mathrm{can}}}$ is nothing but the Fubini-Study
  metric. It follows that the line bundle $j^{\ast} \mathcal{O} (1) \rightarrow Y_m$ is ample, and on the regular part $(Y_m)_{\mathrm{reg}}$ of $Y_m$, the
  curvature of the pull back metric $j^{\ast} h_{\mathrm{can}}$ on $j^{\ast} \mathcal{O} ( 1 )$ is positive. Since $\Phi_m :
  X\setminus B_{m} \rightarrow Y_m$ has dense image, there exists a point $x_0 \in X\setminus B_{m}$ such that $y_0 = \Phi_m (x_0) \in (Y_m)_{\mathrm{reg}}$ and the rank of $(\Phi_m)_{\ast}$ at $x_0$ is $\dim
  Y_m$.

  Define a metric $h = e^{- \varphi}$ on $L \rightarrow X$ by local weight
  \[ \varphi = \log (| s_0  |^2 + \cdots +  | s_N |^2) . \]
  Then $\varphi$ admits analytic singularities and the curvature $\mathrm{i} \Theta_{L, h} \geqslant 0$ in the sense of current. Since
  $\Phi_m = j \circ\Psi_m$ and $h = \Phi_m^{\ast} h_{\mathrm{can}} =
  \Psi_m^{\ast} j^{\ast} h_{\mathrm{can}}$, on $X\setminus B_{m}$, one
  has
  \[ \mathrm{i}\Theta_{L, h} = \Psi_{m}^{\ast} \left(\mathrm{i}\Theta_{\mathcal{O}
     (1), j^{\ast}h_{\mathrm{can}}}\right). \]
  Moreover, $\mathrm{i}\Theta_{\mathcal{O} (1), j^{\ast}h_{\mathrm{can}}} > 0$ and
  $\Psi_m$ is a submersion near $x_0$. So the number of positive eigenvalues of
  $\mathrm{i}\Theta_{L, h}$ is $\kappa ( L )$ near $x_0$. It then follows that
  \[ (\mathrm{i}\Theta_{L, h})^{\kappa ( L )} \wedge \omega^{n - \kappa ( L )} > 0
  \]
  near $x_0$. Therefore, by the definition of $\operatorname{nd} ( L, h )$, we get that
  \[ \kappa (L) \leqslant \operatorname{nd} (L, h) . \]
\end{proof}
	
\section{$L^2$-estimates}\label{sect: l2}

In this section, we prove an $L^2$-estimate which is inspired by \cite{cao2014}. We need the following

\begin{lemma}
  \label{lem:compare}Let $X$ be a complex manifold of dimension $n$, and let $\omega$, $\gamma$ be Hermitian metrics on $X$ such
  that $\gamma \geqslant \omega$. Let $(L,h)\rightarrow X$ be a holomorphic Hermitian line bundle on $X$. Then for every $L$-valued $(p, n)$ form $u$, we
  have
  \[ | u |_{h, \gamma}^2 d V_{\gamma} \leqslant | u |^2_{h, \omega} d
     V_{\omega} . \]
  If $\mathrm{i}\Theta \in C^{\infty}_{1, 1} (M, \mathrm{Herm} (L, L))$ is positive, we
  have
  \[ \langle [\mathrm{i}\Theta, \Lambda_{\gamma}]^{- 1} u, u \rangle_{h, \gamma} d
     V_{\gamma} \leqslant \langle [\mathrm{i}\Theta, \Lambda_{\omega}]^{- 1} u, u
     \rangle_{h, \omega} d V_{\omega}, \]
     where $\Lambda_{\gamma}$ and $\Lambda_{\omega}$ are the adjoint operators of $L_{\gamma}=\gamma \wedge\cdot$ and $L_{\omega}=\omega \wedge\cdot$ respectively.
\end{lemma}

\begin{proof}
The proof is quite standard. For the reader's convenient, we provide the details here.
  Let $x_0 \in M$ be an arbitrary point and choose a coordinate chart $(U, z)
  $ centered at $x_0$ such that
  \[ \omega = \mathrm{i} \sum_{1 \leqslant j \leqslant n} d z_j \wedge d \bar{z}_j, \quad
     \gamma = \mathrm{i} \sum_{1 \leqslant j \leqslant n} \gamma_j d z_j \wedge d
     \bar{z}_j \]
  at $x_0$, where $\gamma_j \geqslant 1$, $j = 1, 2, \cdots, n$ are the
  eigenvalues of $\gamma$ with respect to $\omega$. For a multi-index $K$, let
  $\gamma_K = \prod_{j \in K} \gamma_j$. For any $L$-valued $(p, n)$ form $u =
  \sum_K u_K d z_K \wedge d \bar{z} \otimes e$ with $e$ the local frame of
  $L$ and $d \bar{z} = d \bar{z}_1 \wedge \cdots \wedge d \bar{z}_n$, $| K | =
  p$, arranged in increasing order,  one has
  \[ | u |_{h, \gamma}^2 d V_{\gamma} = \sum_K \frac{1}{\gamma_K} | u_K |^2  |
     e |^2_h d V_{\omega} \leqslant \sum_K | u_K |^2  | e |^2_h d V_{\omega} = | u
     |^2_{h, \omega} d V_{\omega} . \]
  We extend the definition of $u_{jI}$ to non-increasing multi-indices $jI$ by deciding that $u_{jI}=0$ if $jI$ contains identical components repeated and $u_{jI}$ is alternate in $jI$. 
   Let us write
    \[ \mathrm{i}\Theta = \mathrm{i} \sum_{j k} c_{j \bar{k}} dz_{j} \wedge d\bar{z}_{k} \]
    and 
    \[ \widehat{d \bar{z}_j}=d\bar{z}_1\wedge\cdots\wedge d\bar{z}_{j-1}\wedge d\bar{z}_{j+1}\wedge d\bar{z}_n. \]
  Then one can check that
  \[ \Lambda_{\gamma} u = \mathrm{i} \sum_{j, | I | = p - 1} (- 1)^{p + j - 1}
     \frac{1}{\gamma_j} u_{j I} d z_I \wedge (\widehat{d \bar{z}_j}) \otimes
     e, \]
  \[ [\mathrm{i}\Theta, \Lambda_{\gamma}] u = \sum_{| I | = p - 1} \sum_{j, k}
     \frac{1}{\gamma_k} c_{j \bar{k}} u_{k I} d z_{j I} \wedge d \bar{z}
     \otimes e, \]
 and thus
  \begin{eqnarray*}
    \langle [\mathrm{i}\Theta, \Lambda_{\gamma}] u, u \rangle_{h, \gamma} & = &
    \frac{1}{\gamma_1 \gamma_2 \cdots \gamma_n} \sum_{| I | = p - 1}
    \frac{1}{\gamma_I} \sum_{j, k} \frac{1}{\gamma_j \gamma_k} c_{j \bar{k}}
    u_{k I} \bar{ u}_{j I}\\
    & \geqslant & \gamma_1 \gamma_2 \cdots \gamma_n \langle [\mathrm{i}\Theta,
    \Lambda_{\omega}] S_{\gamma} u, S_{\gamma} u \rangle_{h, \omega},
  \end{eqnarray*}
  where $S_{\gamma} u = \sum \frac{1}{\gamma_1 \gamma_2 \cdots \gamma_n
  \gamma_K} u_K d z_K \wedge d \bar{z} \otimes e$.
Therefore,
  \begin{eqnarray*}
    | \langle u, v \rangle_{h, \gamma} |^2 & = & | \langle u, S_{\gamma} v
    \rangle_{h, \omega} |^2\\
    & \leqslant & \langle [\mathrm{i}\Theta, \Lambda_{\omega}]^{- 1} u, u \rangle_{h,
    \omega} \langle [\mathrm{i}\Theta, \Lambda_{\omega}] S_{\gamma} v, S_{\gamma} v
    \rangle_{h, \omega}\\
    & \leqslant & \frac{1}{\gamma_1 \gamma_2 \cdots \gamma_n} \langle [i
    \Theta, \Lambda_{\omega}]^{- 1} u, u \rangle_{h, \omega} \langle [i
    \Theta, \Lambda_{\gamma}] v, v \rangle_{h, \gamma}.
  \end{eqnarray*}
Let  $v = [\mathrm{i}\Theta, \Lambda_{\gamma}]^{- 1} u$, it follows that
  \[ \langle [\mathrm{i}\Theta, \Lambda_{\gamma}]^{- 1} u, u \rangle_{h, \gamma} d
     V_{\gamma} \leqslant \langle [\mathrm{i}\Theta, \Lambda_{\omega}]^{- 1} u, u
     \rangle_{h, \omega} d V_{\omega} . \]
\end{proof}

\begin{proposition}
  \label{prop:l2-estimate}Let $(X,\omega)$ be a compact K\"ahler manifold of dimension $n$, and let $L \rightarrow X$ be a holomorphic line
  bundle over $X$ equipped with a singular metric $h$, which is smooth outside a subvariety $Z \subset X$. Assume
  further that
  \[ \mathrm{i}\Theta_{L, h} \geqslant - \varepsilon \omega \]
  on $X\setminus Z$. Let $\lambda_1 \leqslant \lambda_2 \leqslant \cdots \leqslant \lambda_n$ be
  the eigenvalues of $\mathrm{i}\Theta_{L, h}$ with respect to $\omega$. Then for any $L$-valued $(p, n)$ form $f$ with
  \[ \int_X | f |^2_{h, \omega} d V_{\omega} < \infty , \]
   there exists $u$ and $v$ such
  that
  \[ f = \bar{\partial} u + v, \]
  and
  \[ \int_X | u |^2_{h, \omega} d V_{\omega} + \frac{1}{2 p \varepsilon}
     \int_X | v |^2_{h, \omega} d V_{\omega} \leqslant \int_X
     \frac{1}{\lambda_1 + \lambda_2 + \cdots + \lambda_p + 2 p \varepsilon} |
     f |^2_{h, \omega} d V_{\omega} . \]
\end{proposition}

\begin{proof}
 By {\cite{demailly1982}}, $X\setminus Z$ admits a
  complete K{\"a}hler metric $\omega_1$.
  Set $\omega_{\delta} = \omega
  + \delta \omega_1$. It is a complete K{\"a}hler metric on $X\setminus Z$ for every $\delta>0$. Set $\langle \! \langle \cdot, \cdot \rangle \! \rangle_{h, \delta}$ the global inner product
  defined by $h$ and $\omega_{\delta}$, and $\| \cdot \|_{h, \delta}$ the
  norm induced by this inner product. The completion of the space of smooth, compact-supported $L$-valued $(p,q)$ forms with respect to the norm  $\| \cdot \|_{h,\omega_{\delta}}$ is  $L^2_{(p,q)}(X\setminus Z,L,h,\omega_{\delta})$, whose elements are $L^2$ integrable $L$-valued $(p,q)$ forms with measurable coefficients. Therefore, the operator $\bar{\partial}$ can be extended to an unbounded operator from $L^2_{(p,q)}(X\setminus Z,L,h,\omega_{\delta})$ to $L^2_{(p,q+1)}(X\setminus Z,L,h,\omega_{\delta})$. Here $u\in\operatorname{Dom}\bar{\partial}$ if $u$ and $\bar{\partial}u$ are square integrable on $X$. Denote by $\bar{\partial}^{\ast}$ the adjoint operator of $\bar{\partial}$ and set $\operatorname{Dom}\bar{\partial}^{\ast}$ the domain of $\bar{\partial}^{\ast}$. There exists an orthogonal decomposition
  \[
  L^2_{(p,q)}(X\setminus Z,L,h,\omega)=\operatorname{Ker}\bar{\partial}\oplus \overline{\operatorname{Im}\bar{\partial}^{\ast}}.
  \]

  Let $s$ be an $L$-valued $(p, n)$ form with coefficients in $C^{\infty}_c (X\setminus Z)$.
  Then the Bochner-Kodaira-Nakano equality reads
  \[ \| \bar{\partial}^{\ast} s \|^2_{h, \delta} + \|\bar{\partial} s \|^2_{h, \delta} = \langle \! \langle [\mathrm{i}\Theta_{L, h},
     \Lambda_{\omega_{\delta}}] s, s \rangle \! \rangle_{h, \delta} + \| D's \|^2_{h,
     \delta} + \|D'^{\ast} s \|^2_{h, \delta} \]
  where $D'$ is the $(1,0)$ part of the Chern connection. Since  $\lambda_1 \leqslant \lambda_2 \leqslant \cdots \leqslant \lambda_n$ are
  the eigenvalues of $\mathrm{i}\Theta_{L, h}$ with respect to $\omega$, it follows that
  \[
  \| \bar{\partial}^{\ast} s \|^2_{h, \delta} \geqslant \int_{X\setminus Z} (\lambda_1 +
    \lambda_2 + \cdots + \lambda_p) \left| s \right|^2_{h, \delta} d
    V_{\omega_\delta},
  \]
  therefore,
  \[
   \|\bar{\partial}^{\ast} s \|_{h, \delta}^2 + 2 p \varepsilon \| s \|^2_{h,
    \delta}\geqslant\int_{X\setminus Z} (\lambda_1 +
    \lambda_2 + \cdots + \lambda_p + 2 p \varepsilon) \left| s \right|^2_{h, \delta} d
    V_{\omega_\delta}.
  \]
  Since the space of smooth, compact-supported $L$-valued $(p, n)$ forms  is dense in $L^2_{(p,n)}(X\setminus Z,L,h,\omega_{\delta})$, the above inequality holds true for $s\in L^2_{(p,n)}(X\setminus Z,L,h,\omega_{\delta})$. Therefore, if $s \in \operatorname{Dom}\bar{\partial}^{\ast}\subset L^2_{(p,n)}(X\setminus Z,L,h,\omega_{\delta})$, one has  $s\in \operatorname{Ker}\bar{\partial}$. Since $f \in \operatorname{Ker} \bar{\partial}$, one can obtain
  \begin{eqnarray*}
    &  & \left| \langle \! \langle f, s \rangle \! \rangle_{h, \delta} \right|^2\\
    & \leqslant & \int_{X\setminus Z} \frac{1}{\lambda_1 + \lambda_2 + \cdots + \lambda_p
    + 2 p \varepsilon} \left| f \right|^2_{h, \delta} d V_{\omega_\delta} \cdot \int_{X\setminus Z} (\lambda_1 +
    \lambda_2 + \cdots + \lambda_p + 2 p \varepsilon) \left| s \right|^2_{h, \delta} d
    V_{\omega_\delta}\\
    & \leqslant & \left( \int_{X\setminus Z} \frac{1}{\lambda_1 + \lambda_2 + \cdots + \lambda_p
    + 2 p \varepsilon} \left| f \right|^2_{h, \delta} d V_{\omega_\delta} \right) \cdot \left(\|
    \bar{\partial}^{\ast} s \|_{h, \delta}^2 + 2 p \varepsilon \| s \|^2_{h,
    \delta}\right) .
  \end{eqnarray*}
  Thanks to the Hahn-Banach theorem and the Riesz representation theorem, there
  exists $u_{\delta}$ and $v_{\delta}$ such that
  \[ \langle \! \langle f, s \rangle \! \rangle_{h, \delta} = \langle \! \langle u_{\delta},
     \bar{\partial}^{\ast} s \rangle \! \rangle_{h, \delta} + \langle \! \langle v_{\delta}, s
     \rangle \! \rangle_{h, \delta},\]
       and
       \begin{equation} \label{eq:prop-est-del}
       	\| u_{\delta} \|^2_{h, \delta} + \frac{1}{2 p \varepsilon} \| v_{\delta}
     \|^2_{h, \delta} \leqslant \int_{X\setminus Z} \frac{1}{\lambda_1 + \lambda_2 + \cdots
     + \lambda_p + 2 p \varepsilon} | f |^2_{h, \delta} d V_{\omega_\delta} . 
     \end{equation}
  Therefore
  \[ f = \bar{\partial} u_{\delta} + v_{\delta}.\]
  It follows from Lemma \ref{lem:compare} and the inequality \eqref{eq:prop-est-del} that
  \[ \| u_{\delta} \|^2_{h, \delta} + \frac{1}{2 p \varepsilon} \| v_{\delta}
     \|^2_{h, \delta} \leqslant \int_{X\setminus Z} \frac{1}{\lambda_1 + \lambda_2 + \cdots
     + \lambda_p + 2 p \varepsilon} | f |^2_{h, \omega} d V_{\omega}, \]
  which implies that $\{ u_{\delta} \}_{\delta}$ and $\{ v_{\delta}
  \}_{\delta}$ are bounded in $L^2$ norms on every compact subset of ${X\setminus Z}$. Thus
  there are subsequences of $\{ u_{\delta} \}_{\delta}$ and $\{ v_{\delta}
  \}_{\delta}$ which weakly converge to $u$ and $v$ such that
  \[ f = \bar{\partial} u + v \]
  and
  \[ \| u \|^2_{h, \omega} + \frac{1}{2 p \varepsilon} \| v \|^2_{h, \omega}
     \leqslant \int_{X\setminus Z} \frac{1}{\lambda_1 + \lambda_2 + \cdots + \lambda_p + 2
     p \varepsilon} | f |^2_{h, \omega} d V_{\omega} . \]
By the  extension theorem in  \cite[Lemma 11.10]{Dem12}, we complete the proof.
\end{proof}

Guan-Zhou's solution to Demailly's strong openness conjecture \cite{guan-zhou2015soc} plays a critical role in the proof of Theorem \ref{thm:main}.
\begin{theorem}[{\cite[Theorem 1.1]{guan-zhou2015soc}}]\label{thm:soc}
Let $\varphi$ be a negative plurisubharmonic function on the unit polydisc $\Delta^n\subset \mathbb{C}^n$. Suppose $f$ is a holomorphic function on $\Delta^n$, which satisfies
\[
\int_{\Delta^n}|f|^2e^{-\varphi}d V_n<+\infty,
\]
where $dV_n$ is the Lebesgue measure on $\mathbb{C}^n$. Then for $r\in (0,1)$, there exists $s>0$ such that
\[
\int_{\Delta^n_r}|f|^2e^{-(1+s)\varphi}d V_n<+\infty,
\]
where $\Delta^n_r=\{z\in\mathbb{C}^n:|z_k|<r,k=1,\cdots,n\}$.
\end{theorem}

 It follows from strong openness property (Theorem \ref{thm:soc}) and \cite[Lemma 5.9]{cao2014} that a singular metric for a holomorphic line bundle admits a good regularization.

%

\begin{lemma} [{\cite[Lemma 5.9]{cao2014}}, {\cite[Theorem 1.1]{guan-zhou2015soc}}]
\label{prop:approximation}Let $(L, h)$ be a pesudo-effective line
bundle over a compact K{\"a}hler manifold $(X, \omega)$ of dimension $n$ and
$p \geqslant n - \operatorname{nd} (L, h) + 1$ be an integer. Fix an arbitrary smooth Hermitian metric $h_0$ of $L$, and denote by $\mathrm{i} \Theta_{L,h_0}$ the curvature of $h_0$. Let $h=h_0e^{-2\varphi}$, with $\varphi$ a quasi-plurisubharmonic function on $X$, such that $\mathrm{i} \Theta_{L,h_0}+\mathrm{i}\partial\bar{\partial}\varphi\geqslant 0$ in the sense of currents. Then there exists
a sequence of functions $\{ \hat{\varphi}_k \}_{k = 1}^{\infty}$  on $X$ satisfying the following properties.
\begin{enumerate}[(a)]
\item $\mathcal{I} (\hat{\varphi}_k) = \mathcal{I} (\varphi)$ for all $k \in \mathbb{N}$.
\item 
$\hat{\varphi}_k \leqslant 0$ on $X$.
\item \label{item:prop-app-3}
Let $\lambda_{1, k} \leqslant \lambda_{2, k} \leqslant \cdots
\leqslant \lambda_{n.k}$ be the eigenvalues of $\mathrm{i}\Theta_{L,
h_0e^{-2	\hat{\varphi}_k}}$ with respect to the base metric $\omega$. Then there
exist two sequences $\tau_k \rightarrow 0$ and $\varepsilon_k \rightarrow
0$ such that
\[ \varepsilon_k \gg \tau_k + \frac{1}{k}, \quad \lambda_{1, k} (x) \geqslant -
\varepsilon_k - \frac{C}{k} - \tau_k \]
for all $x \in X$ and $k \in \mathbb{N}$, where $C$ is a constant
independent of $k$.


\item \label{item:prop-app-4}
We can choose $\beta > 0$ and $0 < \alpha < 1$ independent of $k$
such that for every $k$, there exists an open subset $U_k$ of $X$
satisfying
\[ \operatorname{vol} (U_k) \leqslant \varepsilon^{\beta}_k \]
and
\[ \lambda_{p, k} + 2 \varepsilon_k \geqslant \varepsilon_k^{\alpha} \]
on $X\setminus U_k$.
\end{enumerate}
\end{lemma}

\begin{lemma}[{\cite[Lemma 5.10]{cao2014}}]
\label{prop:control}Let $\varphi$ and $\hat\varphi_k$ be as in Lemma \ref{prop:approximation}. Let $f$ be a local section of $\mathcal{O}(\Omega_X^p \otimes L) \otimes \mathcal{I} (\varphi)$
over an open subset $V \subset X$. Then for any relatively compact open subset $U\Subset V$, there exists $s>0$ such that for $k$ large enough,
\[ \int_U | f |^2 e^{- 2 \hat{\varphi}_k} dV_{\omega} \leqslant C_{\| f \|_{L^{\infty}}} \cdot
\left( \int_U | f |^2 e^{- 2 (1 + s) \varphi} dV_{\omega} \right)^{1 / (1 + s)}, \]
where $C_{\| f \|_{L^{\infty}}}$ is a constant depending only on $\| f \|_{L^{\infty}}$.
\end{lemma}
By Guan-Zhou's solution to Demailly's strong openness conjecture (Theorem \ref{thm:soc}), we may assume that $|f|^2e^{-2(1+s)\varphi}$ in Lemma \ref{prop:control} is integrable on $U$ with respect to $d V_{\omega}$.
\begin{theorem}
\label{prop:l2-estimate1}Let $(X, \omega)$ be a compact K{\"a}hler manifold of dimension $n$, and let $(L, h) \rightarrow X$ be a
holomorphic line bundle equipped with a singular metric $h$ satisfying $\mathrm{i}\Theta_{L, h} \geqslant 0$
in the sense of currents. Denote by $\operatorname{nd} (L, h)$ the numerical dimension of
$(L, h)$, and let $p \geqslant n -\operatorname{nd}(L, h) + 1$ be an integer. Fix a
smooth metric $h_0$ on $L$ and write $h = h_0 e^{- 2 \varphi}$. Let $\{
\hat{\varphi}_k \}_{k = 1}^{\infty}$ be the sequence of functions in
Lemma \ref{prop:approximation}. Then for any $f \in L^2_{p, n} (X, L,
h, \omega)$, there exists $u_k$ and $v_k$ such
that
\[ f = \bar{\partial} u_k + v_k \]
with
\begin{equation} \label{equ:l2-2}
	\begin{split} 
	& \int_X | u_k |^2_{h_0, \omega} e^{- 2 \hat{\varphi}_k} d V_{\omega} + \frac{1}{2 p \varepsilon_k} \int_X | v_k |^2_{h_0, \omega} e^{- 2 \hat{\varphi}_k} d V_{\omega} \\
	\leqslant & \int_X \frac{1}{\lambda_{1, k} + \lambda_{2, k} + \cdots + \lambda_{p, k} + 2 p \varepsilon_k} | f |^2_{h_0, \omega} e^{- 2 \hat{\varphi}_k} d V_{\omega}, 
	\end{split}
\end{equation}
and
\begin{equation}
\lim_{k \rightarrow \infty} \int_X | v_k |^2_{h_0, \omega} e^{- 2
	\hat{\varphi}_k} d V_{\omega} = 0. \label{equ:l2-3}
\end{equation}
\end{theorem}

\begin{proof}
By Theorem \ref{thm:soc} and the compactness of $X$, there exists $s > 0$ such that
\[ \int_X | f |^2_{h_0, \omega} e^{- 2 (1 + s) \varphi} d V_\omega< \infty . \]
It follows from Lemma \ref{prop:control} that for any open set $U$,
\begin{equation}\label{equ:control1}
 \int_U | f |^2_{h_0, \omega} e^{- 2 \hat{\varphi}_k} dV_{\omega} \leqslant C_{\| f
	\|_{L^{\infty}}} \cdot \left(\int_U | f |^2_{h_0, \omega} e^{- 2 (1 + s) \varphi} d
V_{\omega} \right)^{\frac{1}{1+s}}.
\end{equation}
From the second condition of Lemma \ref{prop:approximation}, we may
assume $\lambda_{1, k} + 2 \varepsilon_k > 0$ and by Proposition
\ref{prop:l2-estimate}, there exists $u_k$ and $v_k$ such that
\[ f = \bar{\partial} u_k + v_k \]
and
\begin{eqnarray*}
	&  & \int_X | u_k |^2_{h_0, \omega} e^{- 2 \hat{\varphi}_k} d V_{\omega} +
	\frac{1}{2 p \varepsilon_k} \int_X | v_k |^2_{h_0, \omega} e^{- 2
		\hat{\varphi}_k} d V_{\omega}\\
	& \leqslant & \int_X \frac{1}{\lambda_{1, k} + \lambda_{2, k} + \cdots +
		\lambda_{p, k} + 2 p \varepsilon_k} | f |^2_{h_0, \omega} e^{- 2
		\hat{\varphi}_k} d V_{\omega}.
\end{eqnarray*}
It remains to verify that
\[ \lim_{k \rightarrow \infty} \int_X | v_k |^2_{h_0, \omega} e^{- 2
	\hat{\varphi}_k} d V_{\omega} = 0. \]
From the estimate \eqref{equ:l2-2}, we have
\begin{eqnarray*}
	&  & \int_X | u_k |^2_{h_0, \omega} e^{- 2 \hat{\varphi}_k} d V_{\omega} +
	\frac{1}{2 p \varepsilon_k} \int_X | v_k |^2_{h_0, \omega} e^{- 2
		\hat{\varphi}_k} d V_{\omega}\\
	& \leqslant & \int_X \frac{1}{\lambda_{1, k} + \lambda_{2, k} + \cdots +
		\lambda_{p, k} + 2 p \varepsilon_k} | f |^2_{h_0, \omega} e^{- 2
		\hat{\varphi}_k} d V_{\omega}\\
	& \leqslant & \int_{X\setminus U_k} \frac{1}{\varepsilon_k^{\alpha}} | f
	|^2_{h_0, \omega} e^{- 2 \hat{\varphi}_k} d V_{\omega} + \int_{U_k}
	\frac{1}{C_1 \varepsilon_k} | f |^2_{h_0, \omega} e^{- 2 \hat{\varphi}_k}
	d V_{\omega},
\end{eqnarray*}
where the constant $C_1$ is independent of $k$ and the second inequality follows from the properties (c) and (d) of
Lemma \ref{prop:approximation}. Therefore one has
\begin{eqnarray*}
	 \int_X | v_k |^2_{h_0, \omega} e^{- 2 \hat{\varphi}_k} d
	V_{\omega}
	 \leqslant & 2 p \varepsilon_k^{1 - \alpha} \int_{X\setminus U_k} | f
	|^2_{h_0, \omega} e^{- 2 \hat{\varphi}_k} d V_{\omega} + \frac{2 p}{C_1}
	\int_{U_k} | f |^2_{h_0, \omega} e^{- 2 \hat{\varphi}_k} d V_{\omega}.
\end{eqnarray*}
Since $\operatorname{vol }(U_k) \rightarrow 0$ as $k \rightarrow \infty$, it follows from the inequality \eqref{equ:control1} that the second term of the above inequality tends to zero as
$k \rightarrow \infty$. Again by inequality \eqref{equ:control1}, the first term of the above inequality tends to zero as $k
\rightarrow \infty$ since $\varepsilon_k \rightarrow 0$ and $\alpha \in (0, 1)$.
\end{proof}

\section{{\v C}ech cohomology}
\label{sect: cech}
Let $X$ be a compact complex manifold of dimension $n$, and let $L \rightarrow X$ be a holomorphic line bundle equipped with a singular metric $h$ such that $\mathrm{i} \Theta_{L, h} \geqslant 0$ in the sense of currents.
Denote by $\mathcal{U} = \{ U_{\alpha} \}_{\alpha \in I}$ a Stein covering of
$X$ and set
\[ U_{\alpha_0 \alpha_1  \cdots \alpha_q} = U_{\alpha_0} \cap U_{\alpha_1}
   \cap \cdots \cap U_{\alpha_q} . \]
Denote by $\check{C}^q (\mathcal{U}, \mathcal{O}(\Omega^p_X \otimes L) \otimes \mathcal{I} ( h ))$ the
{\v C}ech $q$-cochain of $\mathcal{O}(\Omega^p_X \otimes L) \otimes \mathcal{I} ( h )$. Let $u \in
\check{C}^q (\mathcal{U}, \mathcal{O}(\Omega^p_X \otimes L)\otimes \mathcal{I} ( h ))$ and
$u_{\alpha_0 \alpha_1 \cdots  \alpha_q}$ its component on $U_{\alpha_0
\alpha_1  \cdots  \alpha_q}$. Let
\[ \delta^q : \check{C}^{q} (\mathcal{U}, \mathcal{O}(\Omega^p_X \otimes L) \otimes \mathcal{I}
   (h)) \rightarrow \check{C}^{q+1} (\mathcal{U}, \mathcal{O}(\Omega^p_X \otimes L) \otimes \mathcal{I}
   (h)) \]
be the {\v C}ech operator and $\check{Z}^q \coloneqq \operatorname{Ker} \delta^{q}$.
It is proved in \cite[Lemma 5.8]{cao2014} that $\check{C}^q (\mathcal{U}, \mathcal{O}(\Omega^p_X \otimes L) \otimes \mathcal{I} ( h ))$  is a Fr\'echet space equipped with semi-norms defined by  the $L^2$-integration on relatively compact subsets. The following result is obtained in \cite{cao2014} for $u\in \check{H}^p(X,\mathcal{O}(K_X\otimes L)\otimes \mathcal{I}(\varphi))$. With the same method, it holds true for $u\in \check{H}^q (X, \mathcal{O}(\Omega^p_X \otimes L) \otimes \mathcal{I}
  (h))$ as well. For the sake of completion, we include the proof here.
\begin{lemma}[{\cite[Lemma 5.8]{cao2014}}]
  \label{prop:functional}Let $(X, \omega)$ be a compact complex manifold, and
  let $L \rightarrow X$ be a holomorphic line bundle equipped with a singular metric $h$ such that $\mathrm{i} \Theta_{L, h} \geqslant 0$ in the sense of currents. Let $\{ U_{\alpha} \}_{\alpha \in I}$ be a Stein covering of $X$.
  Let $u$ be an element in $\check{H}^q (X, \mathcal{O}(\Omega^p_X \otimes L) \otimes \mathcal{I}
  (h))$. If there exists a sequence $\{ v_k \} \subset \check{C}^q
  (\mathcal{U}, \mathcal{O}(\Omega^p_X \otimes L) \otimes \mathcal{I} ( h ))$ in cohomology class $u$
  such that
  \[ \lim_{k \rightarrow \infty} \int_{U_{\alpha_0 \alpha_1 \cdots
     \alpha_q}} | v_{k, \alpha_0 \alpha_1 \cdots \alpha_p} |^2_{h_{0,} \omega} d
     V_{\omega} =0, \]
  where $h_0$ in the above equality is a fixed smooth metric on $L$, then $u=0$ in $\check{H}^p (X, \mathcal{O}(\Omega^p_X \otimes L) \otimes \mathcal{I} ( h ))$.
\end{lemma}
 \begin{proof} The proof is due to \cite{cao2014}.
 	Since $\check{C}^q (\mathcal{U}, \mathcal{O}(\Omega^p_X \otimes L) \otimes \mathcal{I} ( h ))$ is a
  Fr{\'e}chet space, the {\v C}ech operator is continuous and its Kernel
  $\check{Z}^p (\mathcal{U}, \mathcal{O}(\Omega^p_X  \otimes L) \otimes \mathcal{I} ( h ))$ is also a
  Fr{\'e}chet space. Therefore the boundary morphism
  \[ \delta^{q-1} : \check{C}^{q - 1} (\mathcal{U}, \mathcal{O}(\Omega^p_X \otimes L) \otimes
     \mathcal{I} ( h )) \rightarrow \check{Z}^q (\mathcal{U}, \mathcal{O}(\Omega^p_X \otimes L )
     \otimes\mathcal{I} ( h )) \]
  is continuous.

  Since $X$ is compact, the cokernel of $\delta^{q-1}$, i.e. $\check{H}^q (X,
  \mathcal{O}(\Omega^p_X \otimes L) \otimes \mathcal{I} ( h ))$, is of finite dimension and thus the
  image of $\delta^{q-1}$ is closed. Thus the quotient morphism
  \[ \operatorname{pr} : \check{Z}^q (\mathcal{U}, \mathcal{O}(\Omega^p_X \otimes L )\otimes \mathcal{I} ( h ))
     \rightarrow \check{H}^q (X, \mathcal{O}(\Omega^p_X \otimes L) \otimes \mathcal{I} ( h )) \]
  is continuous. Therefore the equation
  \[ \lim_{k \rightarrow \infty} \int_{U_{\alpha_0 \alpha_1 \cdots
     \alpha_q}} | v_{k, \alpha_0 \alpha_1 \cdots \alpha_p} |^2_{h_0,\omega} d V_{\omega}
     = 0 \]
  implies $\{ v_k \}_{k = 1}^{\infty}$ tends to $0$ in $\check{Z}^q
  (\mathcal{U}, \mathcal{O}(\Omega^p_X \otimes L) \otimes \mathcal{I} ( h ))$ and hence
  \[ \lim_{k \rightarrow \infty} \operatorname{pr} (v_k) = 0 \in \check{H}^q (X, \mathcal{O}(\Omega^p_X
     \otimes L) \otimes \mathcal{I} ( h )) . \]
  Since $\operatorname{ pr }(v_k)=u \in \check{H}^q (X, \mathcal{O}(\Omega^p_X
     \otimes L) \otimes \mathcal{I} ( h ))$,
      it concludes that
  $u=0$ in $\check{H}^q (X, \mathcal{O}(\Omega^p_X
     \otimes L) \otimes \mathcal{I} ( h ))$.
\end{proof}
%
%

\section{The proof of Theorem \ref{thm:main}}\label{sect:proof}
In this section, we give the proof of Theorem \ref{thm:main}, which we restate as follows.

\begin{theorem}[=Theorem \ref{thm:main}]
  \label{thm:main2}Let $(X,\omega)$ be a compact K\"ahler manifold of dimension $n$. Let $L \rightarrow X$ be a holomorphic line
  bundle equipped with a singular metric $h$ with
  \[ \mathrm{i}\Theta_{L, h} \geqslant 0, \]
  in the sense of currents, and let $\operatorname{nd} (L, h)$ denote the numerical
  dimension of $(L, h)$. Then
  \[ H^n (X, \mathcal O(\Omega^p_X \otimes L) \otimes \mathcal{I} ( h )) = 0, \]
  for $p \geqslant n - \operatorname{nd}(L, h) + 1$.
\end{theorem}

\begin{proof}
Let $\mathcal{L}_{(L, h)}^{p, q}$ be the sheaf of   germs of $L$-valued $(p, q)$ forms
$u$ with measurable coefficients such that both $| u |^2_{h, \omega}$ and $|
\bar{\partial} u |^2_{h, \omega}$ are locally integrable. Consider the following exact sequence
\begin{equation*}
  0 \rightarrow\operatorname{Ker} \bar{\partial}_0 \rightarrow \mathcal{L}_{(L, h)}^{p, 0}
  \xrightarrow{\bar{\partial}_0} \mathcal{L}_{(L, h)}^{p, 1}
  \xrightarrow{\bar{\partial}_1} \cdots \xrightarrow{\bar{\partial} _{n-1}}
  \mathcal{L}_{(L, h)}^{p, n} \longrightarrow 0. 
\end{equation*}
It is well-known that
\[
\operatorname{Ker} \bar{\partial}_0 = \mathcal{O} (\Omega^p_X \otimes L) \otimes \mathcal{I} ( h ),
\]
and
\[ H^q (X, \mathcal{O} (\Omega^p_X \otimes L) \otimes \mathcal{I} ( h )) \cong H^q \left(\Gamma
   \left(X, \mathcal{L}_{(L, h)}^{p, \bullet}\right)\right). \]
Fix a smooth metric $h_{0}$ on $L$, then we can write $h = h_{0} e^{- 2 \varphi}$ for some quasi-plurisubharmonic function $\varphi$ on $X$.
By Lemma \ref{prop:approximation}, one has $\mathcal{I} ( h ) = \mathcal{I} (h_0 e^{- 2
\varphi}) = \mathcal{I} (h_0 e^{- 2 \hat{\varphi}_k})$ for all $k$. Here we use the notations in Lemma \ref{prop:approximation}. Then
\[ H^q (X, \mathcal{O}(\Omega^p_X \otimes L )\otimes \mathcal{I} ( h )) \cong H^q \left( \Gamma \left(
   X, \mathcal{L}_{(L, h_0 e^{- 2 \hat{\varphi}_k})}^{p, \bullet} \right)
   \right), \]
for all $k \in \mathbb{N}$.

Now assume that $p \geqslant n - \operatorname{nd} (L, h) + 1$ and $q = n$. Let
$\mathcal{U} = \{ U_{\alpha} \}_{\alpha \in I}$ be a Stein covering of $X$. By
Proposition \ref{prop:l2-estimate1}, for a $\bar{\partial}$-closed
$L$-valued $(p, n)$ form $f$, there exists $u_k$ and $v_k$ which satisfies the equation $f =
\bar{\partial} u_k + v_k$ and the estimates \eqref{equ:l2-2} and \eqref{equ:l2-3}.

By solving $\bar{\partial}$-equations, we can get a $p$-cocycle
\[ \{ v_{k, \alpha_0 \alpha_1 \cdots \alpha_q} \} \in \check{Z}^q
   (\mathcal{U}, \mathcal{O}(\Omega^p_X \otimes L) \otimes \mathcal{I} ( h )) \]
with the following estimate
\begin{equation} \label{eq:pf:thm:mn:est:cech}
	\int_{U_{\alpha_0 \alpha_1 \cdots \alpha_q}} | v_{k, \alpha_0 \alpha_1
   \cdots \alpha_q} |^2_{h_0, \omega} e^{- 2 \hat{\varphi}_k} d V_{\omega}
   \leqslant C \int_X | v_k |^2_{h_0, \omega} e^{- 2 \hat{\varphi}_k} d
   V_{\omega}, 
\end{equation}
where the constant $C$ is independent of $k$. Moreover, under the de Rham-Weil isomorphism
\[ H^q \left(\Gamma\left(X, \mathcal{L}_{(L, h)}^{p, \bullet}\right)\right) \cong \check{H}^q (X, \mathcal{O}(\Omega^p_X \otimes L) \otimes \mathcal{I} ( h )), \]
the cohomology class of $\{ v_{k, \alpha_0 \alpha_1 \cdots \alpha_q} \}$ in $\check{H}^q (X, \mathcal{O}(\Omega^p_X \otimes L) \otimes \mathcal{I} ( h ))$ corresponds to the cohomology of $v_{k}$ in $H^q \left(\Gamma\left(X, \mathcal{L}_{(L, h)}^{p, \bullet}\right)\right)$.

By properties (b) of Lemma \ref{prop:approximation}, $\hat{\varphi}_k$ is uniformly upper bounded on $X$. Then it follows from the estimate \eqref{eq:pf:thm:mn:est:cech} that
\[ \int_{U_{\alpha_0 \alpha_1 \cdots \alpha_q}} | v_{k, \alpha_0 \alpha_1
   \cdots \alpha_q} |^2_{h_0, \omega} d V_{\omega} \rightarrow 0. \]
Then we can conclude that the cohomology class of $\{ v_{k, \alpha_0 \alpha_1 \cdots \alpha_q} \}$ in $\check{H}^q (X, \mathcal{O}(\Omega^p_X \otimes L) \otimes \mathcal{I} ( h ))$ is zero by Lemma \ref{prop:functional}. By the de Rham-Weil isomorphism, the cohomology of $v_{k}$ in $H^q \left(\Gamma\left(X, \mathcal{L}_{(L, h)}^{p, \bullet}\right)\right)$ is also zero. From the equation
\[ f = \bar{\partial} u_k + v_k, \]
we know that $f$ and $v_{k}$ are in the same cohomology class and hence the cohomology class of $f$ in $H^n (X, \mathcal{O}(\Omega^p_X \otimes L )\otimes \mathcal{I} ( h ))$ is zero. We thus proved that
\[0= H^n \left(\Gamma \left(X, \mathcal{L}_{(L, h)}^{p, \bullet}\right)\right) \cong H^n (X, \mathcal{O}(\Omega^p_X \otimes L) \otimes \mathcal{I} ( h )),  \]
for $p \geqslant n - \operatorname{nd} (L, h) + 1$.
\end{proof}

 Combining the Lemma \ref{lem:ki}, one may find Theorem \ref{thm:main} implies the Bogomolov vanishing theorem.

\begin{theorem}[Bogomolov]\label{thm:bogomolov}
  Let $X$ be a compact K\"ahler manifold of dimension $n$ and $L \rightarrow X$ a
  holomorphic line bundle over $X$. Denote by $\kappa (L)$ the Kodaira
  dimension of $L$. Then
  \[ H^n \left(X, \mathcal{O}\left(\Omega^p_X \otimes L\right)\right) = 0, \]
  for $p \geqslant n - \kappa (L) + 1$.
\end{theorem}

\begin{proof}
  Let $h$ be the singular metric on $L$ constructed in Lemma \ref{lem:ki}. In this case, we have $\kappa ( L ) \leqslant \operatorname{nd} ( L , h )$. Consider the following exact sequence
  \[
  \begin{split}
     \cdots \rightarrow H^n (X, \mathcal{O} (\Omega^p_X \otimes L) &\otimes
    \mathcal{I} \left(h\right))  \rightarrow H^n (X, \mathcal{O}(\Omega_X^p \otimes L))  \rightarrow H^n \left(V (\mathcal{I} (h)), \mathcal{O}(\Omega^p_n \otimes L) |_{V (\mathcal{I}
    (h))}\right) \rightarrow \cdots,
  \end{split}
\]
  where $V ( \mathcal{I} (h))$ is the complex analytic subspace defined by $\mathcal{I} (h)$. By  Theorem  \ref{thm:main}, we have
  \[ H^n (X, \mathcal{O} (\Omega^p_X \otimes L) \otimes \mathcal{I} (h)) = 0. \]
  for $p \geqslant n - \operatorname{nd} ( L , h ) + 1$. However,
  \[ H^n \left(V (\mathcal{I} ( h )), \mathcal{O}\left.\left(\Omega^p_X \otimes L\right) \right|_{V (\mathcal{I} ( h ))}\right) = 0 \]
  since $\dim V (\mathcal{I} ( h ))<n$. It follows that
  \[ H^n \left(X, \mathcal{O}\left(\Omega_X^p \otimes L\right)\right) = 0 \quad \text{for}  \quad p \geqslant n -  \operatorname{nd} ( L , h ) + 1. \]
  By construction, $\kappa ( L ) \leqslant \operatorname{nd} ( L , h )$, so we can conclude that
  \[ H^n \left(X, \mathcal{O}\left(\Omega_X^p \otimes L\right)\right) = 0 \quad \text{for}  \quad p \geqslant n - \kappa (L) + 1. \]
\end{proof}

\end{document}